\documentclass[matann2]{svjour1}

\usepackage{amsmath}
\usepackage{amsfonts}
\usepackage{latexsym}
\usepackage{mathrsfs}
\usepackage{array}
\usepackage{xcolor}
\usepackage{soul}
\usepackage{amscd}
\usepackage{a4}
\usepackage[mathcal]{euscript}
\usepackage{amssymb}
\usepackage{epsf, epic, eepic} 
\usepackage[alphabetic]{amsrefs} 
\usepackage{theorem}
\usepackage{enumerate}

\usepackage{times}      
\usepackage{txfonts}    
\usepackage{tikz}
\usetikzlibrary{arrows.meta}
\usetikzlibrary{decorations.pathmorphing}
\usetikzlibrary{positioning}
\usepackage{graphicx}
\usepackage[T1]{fontenc}

\spnewtheorem*{remark*}{Remark}{\it}{\rm}
\spnewtheorem*{remarks*}{Remarks}{\it}{\rm}
\spnewtheorem*{firstproof*}{First proof}{\it}{\rm}
\spnewtheorem*{secondproof*}{Second proof}{\it}{\rm}
\newcommand{\FSop}{\operatorname{FS^{op}}}
\newcommand{\FS}{\operatorname{FS}}
\newcommand{\rmod}{\vrule width 0mm height 0 mm depth
  0mm_R\mathbf{Mod}}
\DeclareMathAlphabet{\ams}{U}{msb}{m}{n}
\DeclareMathAlphabet{\goth}{U}{euf}{m}{n}
\def\Z{\ams{Z}}
\def\N{\ams{N}}

\def\Q{\ams{Q}}
\def\symn{S_{\kern-1pt n}}
\def\sym5{S_{\kern-1pt 5}}
\def\ss{\sigma}
\def\tt{\tau}
\def\ra{\rightarrow}
\def\ov{\overline}
\DeclareRobustCommand{\stirling}{\genfrac\{\}{0pt}{}}

\def\sgn{\mathrm{sgn}}
\def\Sp{\mathrm{Sp}}
\def\partition{\,\,\begin{tikzpicture}
\draw[line width=0.7pt](0,0)--(0,0.2);
\draw[line width=0.7pt](0,0.1)--(0.175,0.1);
\end{tikzpicture}\,}

\newcommand{\Sym}[1]{S_{\kern-1pt {#1}}}
\title{Homology of matching complexes and representations
of symmetric groups} 

\author{Michael Bate 
$\cdot$
Brent Everitt 
$\cdot$
Sam Ford
$\cdot$
Eric Ramos
\thanks{The authors thank Steve Donkin for suggesting 
  the original problem of this paper.
  The first two authors also thank Chris 
  Bowman-Scargill and Harry Geranios for helpful discussions.
The fourth author would like to send thanks to Brendon Rhoades for
some useful suggestions during the writing of this work
and was supported by NSF grants DMS-2452031 and DMS-2137628.}
}

\institute{Michael Bate,
  Brent Everitt,
  Sam Ford
  \at
Department of Mathematics, 
University of York, 
York
YO10 5DD, UK. 
\email{michael.bate@york.ac.uk,
brent.everitt@york.ac.uk,
sam.ford7@gmx.co.uk}
\and
Eric Ramos
\at
Department of Mathematical Sciences,
Stevens Institute of Technology, Hoboken, NJ 07030, USA.
\email{eramos3@stevens.edu}
}

\titlerunning{Homological representations of symmetric groups}
\authorrunning{Michael Bate, Brent Everitt, Sam Ford, Eric Ramos}

\begin{document}

\maketitle

\begin{abstract}
We compute the homology of the matching complex
$M(\Gamma)$, where $\Gamma$ is the complete hypergraph on $n\geq 2$
vertices, and analyse the $\Sym{n}$-representations carried by this
homology. These results are achieved using standard techniques in
combinatorial topology, such as the theory of shellings.
We then broaden the scope to the larger class of fibre-closed families of
simplicial complexes and
consider these through the lens of representation stability.
This allows us to prove a number of results of an asymptotic nature,
such as an analysis of the growth of Betti numbers and the
kinds of irreducible $\symn$-representations that
appear.
\end{abstract} 

\maketitle


\section*{Introduction}

Let $G$ be a finite group, $p$ a prime dividing the order of $G$ and
$\mathcal{S}_p(G)$ the poset of $p$-subgroups of $G$ ordered 
by inclusion. 
Brown's homological Sylow theorem \cite{MR0385008} states that
the simplicial order complex of $\mathcal{S}_p(G)$ 
has Euler characteristic congruent to $1$ modulo 
the order of a Sylow $p$-subgroup of $G$. This led
Quillen \cite{MR0493916} to show that $\mathcal{S}_p(G)$ is 
homotopy equivalent to
the poset $\mathcal{A}_p(G)$ of elementary 
Abelian $p$-subgroups
of $G$ and to formulate a number of fundamental conjectures.
These results have 
established $\mathcal{S}_p(G)$ and 
$\mathcal{A}_p(G)$
as central objects
in the interplay between geometric combinatorics and
representation theory.

Even for the 
symmetric groups these complexes
are not well-understood. Bouc \cite{MR1174893} thus
introduced the matching complex $M_n$, which 
is closely related to 
$\mathcal{S}_2(\Sym{n})$ and $\mathcal{A}_2(\Sym{n})$.
In particular, if $k$ is a field of characteristic $0$ then the 
natural $\Sym{n}$-action
on $M_n$ gives the reduced homology
$\widetilde{H}_{q}(M_n,k)$ the structure of a $
k\Sym{n}$-module, and Bouc describes its decomposition
into irreducibles. 
On the other hand, the integral homology of $M_n$ displays 
complicated torsion behavior
and is less well-understood.
This has led to much work in recent years
on matching complexes and their various 
generalisations.

This paper starts --- see Theorem \ref{thm:homology} ---
by computing the homology of one of these 
generalisations: the matching complex
$M(\Gamma)$, where $\Gamma$ is the complete hypergraph on $n\geq 2$
vertices. (The original  matching complex $M_n$ is equal to $M(\Gamma)$
where $\Gamma$ is the complete \emph{graph\/} on $n$ vertices.)
Our main technical tool here is the theory of shellings.

We then analyse the resulting $\Sym{n}$-representations carried by the
homology and decompose them (Theorem \ref{thm:homologyinduced} below)
as a sum of induced representations from the normalisers of standard
Young subgroups $\Sym{\lambda}$, for partitions $\lambda\partition n$
having a certain shape. If $\lambda$ has no repeated parts then
this gives a decomposition \eqref{eq:kostka}
of the corresponding part of the homology into Specht modules.
Such a decomposition in general seems to be out of reach, so 
in lieu of this we appeal to the theory of representation stability
and study the asymptotics of the $\symn$-representations
as $n$ grows large. This, it turns out, we can do for a larger
class of simplicial complexes that we term \emph{fibre closed},
and categorical representation theory allows us to
expose the overall ``shape" of the representations; this we do in
Theorem \ref{mainthm}.


\section{Matching complexes and their homology}
\label{section:complexes}

We begin by recalling basic terminology 
and notation for simplicial complexes --- standard references for this 
material are \cite{MR1402473}*{Section 2.3} or 
\cite{MR0210112}*{Section 3.1}.
We then recall the general matching complex $M(\Gamma)$ of a
hypergraph $\Gamma$;
see 
\cite{MR1013569} for background on hypergraphs,
\cite{MR2536865} for matching theory and 
\cite{MR2368284} for a survey of matching and related complexes.
The homology and the representations it carries are given in
Section \ref{subsection:homology}, along with the shellings that we
use to compute it. 
A nice survey of general simplicial methods in
algebra and combinatorics is \cite{Wachs07}; see also \cite{MR2022345}
for representations carried by the homology of matching complexes. 
For shellings, especially in the non-pure case, we follow
\cite{MR1333388}. 

\subsection{Matching complexes}
\label{subsection:matching}

A (finite) simplicial complex $X$ with vertex set $X_0$ is 
a set of distinguished subsets of $X_0$ with the property 
that $\{x\}\in X$ for all $x\in X_0$, and
if $\ss\in X$ and $\tt\subseteq\ss$ then $\tt\in X$. 
If $\ss\in X$ with $|\ss|=q+1$ then 
$\ss$ is called a $q$-(dimensional) face of $X$. 
The empty set
$\varnothing$ is the unique face of dimension $-1$. 
Write $X_q$ for the set of $q$-faces.

A face $\ss$ that is maximal under inclusion is called a 
facet
and a complex where all the facets have the same dimension
is said to be pure. The dimension of $X$ is the maximal $q$
for which there exists a $q$-facet.
A $q$-simplex
$\overline{\ss}$ is the
simplicial subcomplex obtained by considering a $q$-face $\ss$ and
all of its subsets; in particular, if $\ss$ is a 
facet then $\overline{\ss}$ is called
a maximal subsimplex. A map
$f:X\ra Y$ of simplicial complexes is a set map $f:X_0\ra Y_0$ such that
$f(\ss)\in Y$ for all $\ss\in X$. A group $G$ acts on $X$ if $G$ acts on the 
vertices $X_0$ in such a way that for any $g\in G$ and face $\ss$, the image subset 
$g\ss$ is also a face of $X$.

Now to matching complexes.
A hypergraph $\Gamma$ consists of a set of vertices $V$ and a
set of non-empty subsets of $V$ called hyperedges. 
A hypergraph whose hyperedges all have cardinality $\leq 2$ is just a
graph. In particular,
loops at a vertex of a graph are allowed (these are the hyperedges
corresponding to subsets of size $1$)
but multiple edges between
vertices are not (the hyperedges form a set, not a multi-set).
A matching of a hypergraph
$\Gamma$ is some set of pairwise disjoint hyperedges. 
The \emph{matching
complex\/} $M(\Gamma)$ has $q$-faces the matchings of $\Gamma$ that
contain $q+1$ hyperedges. 
Put another way, for $q\geq 0$, the matching complex $M(\Gamma)$ 
has $q$-faces those collections of $q+1$ mutually disjoint 
non-empty subsets of
the vertices $V$ of $\Gamma$ that happen to be hyperedges;
in particular, the vertices of $M(\Gamma)$ are
the non-empty subsets of $V$ that are hyperedges.

\vspace{1em}

The fundamental example is when
$\Gamma$ is the complete graph $K_n$ on $n$ vertices --- where any
two vertices are joined by an edge and no loops are allowed. Then
$M(\Gamma)=M(K_n)$ is \emph{the\/} matching complex
$M_n$. Alternatively, $M_n$
is the simplicial complex whose vertices are the $2$-element subsets
of $[n]:=\{1,2,\ldots,n\}$ and whose 
$q$-faces are the collections of $q+1$ mutually disjoint $2$-element
subsets of $[n]$. 
If instead $\Gamma$ is the complete bipartite graph $K_{m,n}$
then $M(\Gamma)$ is isomorphic to the \emph{chessboard\/} complex,
whose $q$-faces 
can be identified with 
the placements of $q+1$ mutually non-attacking rooks on an
$m\times n$ chessboard.

For an example where $\Gamma$ is a genuine
hypergraph, rather than just a graph, consider
for a fixed $r\geq 2$ the 
$r$-uniform complete 
hypergraph
$\Gamma=K_n^r$, whose hyperedges are all the subsets 
of $[n]$ of size $r$. The
resulting matching complex $M(\Gamma)$ then 
has $q$-faces the $q+1$ mutually disjoint subsets of $[n]$, all of
which have size $r$ --- see \cite{MR1253009}.
When $r=2$ the complex $M(K_n^2)$ is just the
matching complex $M_n$.

\vspace{1em}

Our interest starts with the matching complex $\overline{X}(n):=M(\Gamma)$
where 
$\Gamma$ is 
the complete hypergraph on $n\geq 2$ vertices: 
the hyperedges of $\Gamma$ are now
\emph{all\/} the
non-empty subsets of $[n]$.
The vertices of $\overline{X}(n)$ are the non-empty subsets
$x\subseteq [n]$ and
a $q$-face is a partition $x_0|x_1|\cdots|x_q$ of some subset
$A\subseteq [n]$ into $q+1$ non-empty blocks.
The complex $\ov{X}(n)$ is $(n-1)$-dimensional with a unique $(n-1)$-face
given by the partition of $[n]$ into $n$ blocks of size $1$. 

The complex $\overline{X}(n)$ has two connected components: one
consists of just the vertex
$[n]$, which is contained in no other face. On the other hand,
if $x$ is any other vertex then there is a path of at
most two edges connecting $x$ to the vertex $\{1\}$. For, if $1\not
\in x$ then the partition $1|x$ is an edge connecting $x$ to $\{1\}$; if $1\in
x$ then there is an $i\not\in x$ and hence edges $1|i$ and $i|x$ connecting
$x$ to $\{i\}$ and then $\{i\}$ to $\{1\}$. For this reason, we remove
the isolated vertex $[n]$ 
and write $X(n)$ for the complex whose vertices are the \emph{proper\/}
non-empty subsets $\varnothing\not=x\subsetneq [n]$ and whose
$q$-faces are 
the partitions of some subset $A\subseteq [n]$ into $q+1$ proper non-empty
blocks. 

The facets of $X(n)$ are the partitions $x_0|x_1|\cdots|x_q$ of $[n]$ itself.
In particular, the number of $q$-facets equals the number of partitions
of the set $[n]$ into $q+1$ non-empty blocks; this is given by the Stirling
number of the second kind:
$$
\stirling{n}{q}=
\frac{1}{q!}\sum_{i=0}^{q}(-1)^i\binom{q}{i}(q-i)^n.
$$
Since these numbers are nonzero for multiple values of $q$,
the complex $X(n)$ is not pure. 


\subsection{Homology}
\label{subsection:homology}

Returning to a general simplicial complex $X$, we
fix a total ordering $x<y$ of the vertices. A $q$-face is then oriented
when it is written as 
$\ss=x_0x_1\ldots x_q$ with $x_i<x_j$ iff $i<j$.
The reduced simplicial 
homology $\widetilde{H}_*(X,\Z)$
is the
homology of the chain complex $C_*=C_*(X,\Z)$ having $q$-chains
$C_q$ the free $\Z$-module on the $q$-faces $X_q$. The 
differential $d:C_q\ra C_{q-1}$ is the map induced by
$\ss
\mapsto
\sum_i (-1)^i d_i\ss$, where
$d_i\ss=x_0x_1\ldots \widehat{x}_i\ldots x_q\,(0\leq i\leq q)$,
and $d_0:C_0\ra C_{-1}= \Z$ 
is induced by $x_0\mapsto 1$.

If a group $G$ acts on $X$ and $\ss=x_0x_1\ldots x_q$ is an oriented 
face then: 
$$
g\cdot\ss=g\cdot x_0x_1\ldots x_q=(gx_0)(gx_1)\ldots (gx_q)
=y_{\pi(0)}y_{\pi(1)}\ldots y_{\pi(q)},
$$
where $y_0y_1\ldots y_q$ is an oriented face and $\pi$ is in the 
symmetric group $S_{\kern-1pt q}$. Define an action of $G$ on the $q$-chains
$C_q$ induced by:
\begin{equation}
\label{eq100}
    g\cdot x_0x_1\ldots x_q=(-1)^{\text{sgn}(\pi)} y_0y_1\ldots y_q.
\end{equation}
Then $gd=dg$ and (\ref{eq100}) defines an action of $G$ by chain maps 
on the complex $C_*$. In particular 
there is an induced action of $G$ on the
homology $\widetilde{H}_q(X,\Z)$ for all $q$. 

We compute homology by finding a
\emph{shelling\/} of $X$. This is a 
total ordering $\sigma_1,\sigma_2,\ldots,\sigma_m$ of the
facets such that for all $t>1$:
\begin{equation}
\label{eq5}
    \overline{\sigma}_t
\cap
{\textstyle \bigcup}_{i=1}^{t-1}\,\overline{\sigma}_i,
\end{equation}
is a pure $(\dim \sigma_t-1)$-dimensional subcomplex of 
the simplex $\overline{\sigma}_t$. In this case, define:
\begin{equation}
\label{eq6}
\mathscr{R}(\sigma_t)
=\{x\in{\sigma}_t:
{\sigma}_t\,\setminus\, x\in 
{\textstyle\bigcup}_{i=1}^{t-1}\overline{\sigma}_i\},    
\end{equation}
and call $\ss$ a homology facet when 
$\mathscr{R}(\sigma)=\sigma$. 

The homology
$\widetilde{H}_q(X,\Z)$ is isomorphic to the free $\Z$-module with
a basis in 1-1 correspondence with 
the homology facets
--- see \cite{MR1333388}*{Corollary 4.4}.
If $G$ acts on $X$
by preserving the set of homology facets
then $G$ permutes this basis exactly as it permutes the homology facets
themselves. 
This allows the
$G$-module 
structure of $\widetilde{H}_q(X,\Z)$ to be analysed.

\vspace{1em}

Returning to the matching complex $X(n)$ of Section
\ref{subsection:matching}, 
from now on we abbreviate a
partition $x_0|x_1|\cdots|x_q$ of $[n]$ to just $x_0x_1\ldots x_q$. It
will also be convenient to isolate those blocks of a partition that
are singletons and write $x_0x_1\ldots x_q$ as $x_0\ldots x_kz_1\ldots
z_\ell$, where the blocks $x_i$ have size $\geq 2$ and the blocks
$z_j$ have size $1$.

Totally order the facets of $X(n)$ in decreasing order of the number
of singleton blocks: i.e. define:
$$
x_0\ldots x_k z_1\ldots z_\ell
<
x'_0\ldots x'_u z'_1\ldots z'_v,
$$
if $\ell>v$; if $\ell=v$, then totally order the partitions with 
$\ell$ singleton blocks arbitrarily.

\begin{proposition}
  \label{prop1}
This order is a shelling of $X(n)$.
\end{proposition}

\begin{proof}
Let $\ss_t=x_0\ldots x_k z_1\ldots z_\ell$ be a facet. We claim that
the subcomplex (\ref{eq5}) is the union of all the simplices
$\overline{\tau}$ where:
\begin{equation}
  \label{eq:2}
\tau=x_0\ldots \widehat{x}_i \ldots x_k z_1\ldots z_\ell,
\end{equation}
with $0\leq i\leq k$ and the hat denoting omission. As each such
$\overline{\tau}$ is a $(\dim \sigma_t-1)$-dimensional sub-simplex of
$\overline{\sigma}$, the proposition 
follows. To see the claim, if in (\ref{eq:2}) the block
$x_i=\{x_{i1},\ldots,x_{ij}\}$, then $\tau$ is a face of the facet:
$$
\ss_s=x_0\ldots \widehat{x}_i \ldots x_k z_1\ldots z_\ell x_{i1}\ldots
x_{ij},
$$
and $\ss_s<\ss_t$. Thus $\ov{\tau}\subseteq\ov{\ss}_s\in
  \bigcup_{i=1}^{t-1}\,\overline{\sigma}_i$ and $\ov{\tau}$ is indeed
  contained in the subcomplex (\ref{eq5}).
On the other hand, let $\omega$ be a face of $\ss_t$ of the form:
$$
\omega=x_0\ldots x_k \widehat{z}_1\ldots\widehat{z}_jz_{j+1}\ldots z_\ell.
$$
If $\ss_s$ is a facet that contains $\omega$ as a face then
$\ss_s=
x_0\ldots x_k y_1\ldots y_r z_{j+1}\ldots z_\ell$
where $y_1\ldots y_r$ is some partition of the set
$\{z_1,\ldots,z_j\}$. In particular, $\ss_s\geq \ss_t$ and
$\ov{\omega}\not\in \bigcup_{i=1}^{t-1}\,\overline{\sigma}_i$. This
proves the claim. 
  \qed
\end{proof}

An immediate consequence of the claim at the beginning of the proof of
Proposition \ref{prop1} is that the homology facets of $X(n)$ are
those facets which as 
partitions have no singleton blocks. We have thus proved:

\begin{theorem}
  \label{thm:homology}
For $n\geq 2$ and $0\leq q\leq n-1$, the homology
$\widetilde{H}_q(X(n),\Z)\cong\Z^{\,\beta(n,q)}$ where 
$\beta(n,q)$ is the number of partitions of $[n]$ into $q+1$
blocks, with each block a proper non-singleton subset.
\end{theorem}

In particular if $q=0$ or $q>\lfloor\frac{n}{2}\rfloor-1$, 
where $\lfloor x\rfloor$ is the largest integer $\leq x$, 
then the homology $\widetilde{H}_q(X(n),\Z)$ vanishes.

For an integer $r\geq 1$, the $r$-associated Stirling number
$\stirling{n}{k}_r$ of the
second kind is the number of partitions of $[n]$ into $k$
blocks, each of which has size at least $r$ --- see
\cite{MR0460128}*{page 221}. There are various closed formulas for
$\stirling{n}{k}_r$ --- see \cite{MR0600368} --- and in particular,
inclusion-exclusion gives, for $q>0$:
$$
\beta(n,q)=\stirling{n}{q+1}_2=
\sum_{i=0}^{q+1}(-1)^i\binom{n}{i}\stirling{n-i}{q-i+1}.
$$


\section{Representations of the symmetric group}
\label{section:representations}

%
%
%

We saw in Section \ref{subsection:homology} that we can make the 
homology groups for $X(n)$ into $\Z\symn$-modules: by defining a total 
order on the vertices we can orient each face of $X(n)$, and then by 
adjusting the natural action with some sign changes we can ensure that 
orientations are preserved, so that $\symn$ acts on the homology
groups.
For the fundamentals of symmetric group representations, see
\cite{MR1824028}.

A convenient total order of the vertices of $X(n)$ is to first order 
subsets of $[n]$ in decreasing order of size, and then lexicographically 
order the subsets of size $k$ for each $k\leq n$. 
Thus, any oriented $q$-face $\ss=x_0x_1\ldots x_q \in X(n)$ has 
$|x_0|\geq \cdots \geq |x_q|$.
We denote the corresponding orientation-preserving action by 
$(g,\sigma) \mapsto g(\sigma)$ for $g\in \symn$ and $\sigma\in X(n)$.

Each facet $\sigma\in X(n)$ gives rise to a partition of the 
number $n$: define $\lambda(\sigma):=(|x_0|,\ldots,|x_q|)\partition n$.
We call the partition $\lambda=\lambda(\sigma)$ the \emph{shape} of
$\sigma$. 
For each $\lambda=(\lambda_0,\lambda_1,\ldots,\lambda_q)$, and for 
the rest of
this section,
let $x_0x_1\ldots x_q$ with:
\begin{equation}\label{min_face}
    x_0=\{1,\ldots,\lambda_0\},
x_1=\{\lambda_0+1,\ldots,\lambda_0+\lambda_1\},
\ldots,
x_q=\biggl\{1+\sum_{i=0}^{q-1}\lambda_i,
\ldots,\sum_{i=0}^{q}\lambda_i\biggr\},
\end{equation}
be the minimal (in the total order above) 
  oriented $q$-face of 
shape $\lambda$.

Given a partition 
$\lambda = (\lambda_0,\lambda_1,\ldots,\lambda_{q})\partition n$,
we can form the standard Young subgroup: 
$$
\Sym{\lambda}
\cong
\Sym{\lambda_{0}}
\times \cdots\times
\Sym{\lambda_{q}},
$$
where the factor  
$\Sym{\lambda_i}$
is identified as the permutations of the set 
$x_i$ in the minimal $q$-face (\ref{min_face}).
Let $N_\lambda = N_{\symn}(\Sym{\lambda})$ be the normaliser 
in $\symn$ of $\Sym{\lambda}$. 
If we rewrite $\lambda = \left(\mu_1^{a_1},\ldots,\mu_r^{a_r}\right)$ 
by collecting together all the parts of $\lambda$ with the 
same size, then there is a homomorphism:
$$
N_\lambda \ra N_\lambda/\Sym{\lambda} \cong \Sym{a_1}\times \cdots \times \Sym{a_r},
$$
with $g\mapsto \bar{g}$ for $g\in N_\lambda$. In particular, if 
$x_0x_1\ldots x_q$ is the face (\ref{min_face})
then
$g\cdot x_0x_1\ldots x_q=x_{\bar{g} 0}x_{\bar{g} 1}\ldots x_{\bar{g}q}$.
(In fact, we can identify $N_\lambda$ explicitly inside $\symn$ as a 
direct product of wreath products: for each $1\leq i\leq r$, we 
have a factor with base group the direct product of the $a_i$ copies
of $\Sym{\lambda_i}$ in $\Sym{\lambda}$ and a symmetric group $\Sym{a_i}$
on top permuting those copies.)

With this notation in hand, we can identify the homology groups as 
representations of $\Sym{n}$.
First identify $\widetilde{H}_q(X(n),\Z)$ with 
the free $\Z$-module spanned by the homology facets, where 
a homology facet $\sigma$ 
has partition 
$\lambda(\sigma)\partition n$ with no parts of size $1$. 
For each such partition 
it is clear that the free $\Z$-module spanned by the facets of 
a fixed shape $\lambda$ identifies with a $\Z\symn$-submodule of 
$\widetilde{H}_q(X(n),\Z)$. 
Denote this submodule by $V(\lambda)$.

Let $V=\Z[v]$ be a $1$-dimensional $\Z$-module and define an
$N_\lambda$ action on $V$ by $g\cdot v=\text{sgn}(\bar{g})\, v$.
We then have:

\begin{theorem}
  \label{thm:homologyinduced}
For each partition $\lambda\partition n$
there is an isomorphism of 
$\Z\symn$-modules:
$$
V(\lambda) \cong
V\kern1pt{\uparrow^{\kern1pt\Sym{n}}_{\kern1ptN_\lambda}},
$$
where the right-hand side is the $N_\lambda$ module $V$ induced 
up to $\Sym{n}$,
and hence
$$
\widetilde{H}_q(X(n),\Z) \cong \bigoplus_\lambda 
V\kern1pt{\uparrow^{\kern1pt\Sym{n}}_{\kern1ptN_\lambda}},
$$
as $\Z\symn$-modules,
where the sum is over all $\lambda$ of length $q+1$ and with no parts
of size $1$ or $n$.
\end{theorem}

\begin{proof}
  Let $T=\{t_0=1,t_1,\ldots, t_m\}\in\Sym{n}$ be such that the
  $t_i\cdot x_0x_1\ldots x_q,\, (0\leq i\leq m),$ are the oriented
  $q$-faces of shape $\lambda$, where $x_0x_1\ldots x_q$ is the minimal
  oriented $q$-face from (\ref{min_face}). In particular,
  $T$ is a transversal
  for $N_\lambda$ in $\Sym{n}$ and the induced module
$V\kern1pt{\uparrow^{\kern1pt\Sym{n}}_{\kern1ptN_\lambda}}=\Z[t\otimes
v]_{t\in T}$. 
If
$h\in\Sym{n}$,
then for any $t_i\in T$ we have $ht_i=t_jg$ for
some $t_j\in T,g\in N_\lambda$. 
The induced $\Sym{n}$-action is then given by
$h\cdot(t_i\otimes v)=t_j\otimes g\cdot v=t_j\otimes
\text{sgn}(\bar{g}) v$. On the other hand $V(\lambda)=
\Z[t\cdot x_0x_1\ldots x_q]_{t\in T}$,
and the action
(\ref{eq100}) is:
$$
h\cdot(t_i\cdot x_0x_1\ldots x_q)=t_jg\cdot x_0x_1\ldots x_q
=t_j\cdot x_{\bar{g} 0}x_{\bar{g} 1}\ldots x_{\bar{g}q}
=\text{sgn}(\bar{g})\, t_j\cdot x_0x_1\ldots x_q.
$$
The map $t\cdot x_0x_1\ldots x_q\mapsto t\otimes v$ is then the
$\Z\symn$-isomorphism we seek.
\qed
\end{proof}

\begin{remark}\label{rem:kostka}
Suppose that $\lambda$ is a partition with no repeated parts.
Then $N_\lambda = \Sym{\lambda}$ and  $V$ is 
the trivial module.
The induced module
is thus
the permutation module $M^\lambda$ 
corresponding to the partition $\lambda$.

If we base change to a field of characteristic $0$, then the 
decomposition of this permutation module into simple 
modules for $\symn$ is well-known: 
\begin{equation}\label{eq:kostka}
V\kern1pt{\uparrow^{\kern1pt\Sym{n}}_{\kern1ptN_\lambda}}\cong
M^\lambda\cong\bigoplus_{\mu\,\unrhd\lambda} \kappa_{\mu\lambda}\Sp(\mu),
\end{equation}
where $\Sp(\mu)$ is the Specht module corresponding to
the partition $\mu\partition n$ and the sum is over all $\mu$
that dominate $\lambda$. The coefficient $\kappa_{\mu\lambda}$ is the
number of  
semistandard tableaux of shape $\mu$ and content $\lambda$ (a Kostka
number ---  
see for example \cite{MR1824028}*{\S 2.11}).

If we instead base change to an algebraically closed field of positive 
characteristic, then there is an analogous decomposition of 
the permutation module into a direct sum of Young modules,
see \cite{donkin}*{Lemmas (3.4) and (3.5)}.
\end{remark}

\begin{remark}
Building on the previous remark, in \emph{all} cases we can 
identify the induced module
$V\kern1pt{\uparrow^{\kern1pt\Sym{n}}_{\kern1ptN_\lambda}}$ 
as a submodule of the permutation module $M^\lambda$, as follows.

The module $M^\lambda$ has a basis in bijection with tabloids of 
shape $\lambda$. 
Consider the (standard) tabloid $T$ of shape $\lambda$ filled 
left-to-right, top-to-bottom with the numbers $1,\ldots,n$.
The group $\Sym{\lambda}$ fixes this tabloid, and the group 
$N_\lambda$ acts as permutations of the rows.
Let:
$$
S = \sum_{g\in N_\lambda} \sgn(\bar{g}) (g\cdot T).
$$
Let $V^\lambda$ denote the cyclic $\Z\symn$-submodule of $M^\lambda$ 
generated by $S$ --- that is, $V^\lambda$ is $\Z$-spanned by the 
elements $g\cdot S$ as $g$ runs over the elements of $\symn$.
We see easily that
$V^\lambda \cong V\kern1pt{\uparrow^{\kern1pt\Sym{n}}_{\kern1ptN_\lambda}}$.
(If $g_1,\ldots,g_r$ is a system of coset representatives for $N_\lambda$, 
then $V^\lambda$ is the $\Z$-span of the $g_i\cdot S$, and 
$\symn$ acts precisely as it does on the induced module.)
\end{remark}

\section{Representation stability of fibre-closed families}

Even over a field $k$ of characteristic $0$, a complete decomposition of the 
$k\symn$-module $\widetilde{H}_q(X(n),k)$
into irreducibles 
seems to be out of reach: to analyse the structure of the 
induced modules appearing becomes very hard away from cases
like Remark \ref{rem:kostka} above. For example, the discrete Hodge-theoretic
techniques used in
\cites{MR1912799,MR1648484}
to decompose into irreducibles
the homology
of the classical matching complexes $M_n$ do not carry over to the
complexes $X(n)$.

In lieu of such a decomposition we work in an asymptotic manner
and analyse the overall ``shape'' of the representations in
Theorem \ref{thm:homologyinduced} (or, more precisely, the analogue of Theorem 
\ref{thm:homologyinduced} over $\Q$). We start with two observations:
firstly, the Betti numbers $\beta(n,q)$, for any fixed integer
$q \geq 0$, of $X(n)$ can be expressed
as a $\Q[n]$-linear combination of exponential terms.
Secondly, the Specht modules $\Sp(\lambda)$
appearing in a decomposition of
$\widetilde{H}_q(X_n;\Q)$ have the property that the length 
of $\lambda$ is at most
$q+1$. It turns out that these two
observations, along with a few others that are less
immediately deduced from Theorem \ref{thm:homologyinduced},
remain true across a more general collection of simplicial complexes
than just the $X(n)$ of Sections
\ref{section:complexes}-\ref{section:representations}.

\subsection{Fibre-closed families of simplicial complexes}

Let $\{X_n\}_{n \in \N}$ denote a family of simplicial complexes
such that the vertices of $X_n$ consist of
certain subsets of $[n]=\{1,2,\ldots,n\}$. 
We call such a 
family \emph{fibre closed} when for every surjective function
$f:[a] \rightarrow [b]$, and every face
$\ss = \{x_0,\ldots,x_q\}$ of $X_b$, the preimage
$f^{-1}\ss = \{f^{-1}x_0,\ldots,f^{-1}x_q\}$ is a
face of $X_a$.
The condition of being fibre-closed
implies that each of the simplicial complexes $X_n$
carries an action by $\symn$.

The motivating examples for us are the matching complexes $X(n)$
of the complete hypergraphs on $n$ vertices from Sections
\ref{section:complexes}-\ref{section:representations}.
Other examples include the order complex of the lattice of subsets
of $[n]$ and the \emph{anti-order\/} complex (whose faces comprise the
anti-chains) of this lattice. 

If $m\leq n$ are non-negative integers and 
$\lambda = (\lambda_1,\ldots,\lambda_r)\partition m$ 
is a partition of $m$, write $\lambda[n]$ for the padded partition
$\lambda[n] := (n-m,\lambda_1,\ldots,\lambda_r)$.

For any $d \geq 0$, we define
$\mathbf{X}_d:\bigcup_{n \geq 0}\Sym{n}\rightarrow \Q$ (disjoint union) by setting
$\mathbf{X}_d(\sigma)$ to be the number of $d$-cycles in the cycle decomposition
of $\sigma$. A character polynomial is then defined to be any class
function that can be expressed as a polynomial in the variables $\mathbf{X}_d$.
In a similar vein, if $A = 1^{a_1}2^{a_2}\cdots$ and $\nu = 1^{m_1}2^{m_2}\cdots$
denote partitions of two (not necessarily distinct) integers,
then we can define a class function:
\[
\binom{\mathbf{X}}{\nu}A^{\mathbf{X}-\nu} := \prod_{d \geq 0}
\binom{\mathbf{X}_d}{m_d}\left(\sum_{n \mid d} na_n\right)^{\mathbf{X}_d-m_d}
\]
These particular character functions were introduced in
\cite{Tost1}*{Theorem 1.11}, and will be critical in the next theorem.

Here is the promised asymptotic result, where except for the last part,
we work over $\Q$:
 
\begin{theorem}
  \label{mainthm}
  Let $\{X_n\}_{n \in \N}$ be a fibre-closed family of simplicial
  complexes, and let $q \geq 0$ be fixed.
\begin{enumerate}[(i).]
\item There exist polynomials $f_1,\ldots,f_{2^{q+1}} \in \Q[x]$,
such that the $q$-th Betti number $\beta(n,q)$ of $X_n$ can be
expressed as:
        \[
        \beta(n,q) = f_1(n) + f_2(n)2^n + f_3(n)3^n + \ldots
        + f_{2^{q+1}}(n)(2^{q+1})^n,
        \]
for all $n \gg 0$.
\item The $\Q\symn$-module $\widetilde{H}_q(X_n,\Q)$ decomposes as a sum of
Specht modules $\Sp(\lambda)$ where the partition
$\lambda\partition n$ has length
bounded by $2^{q+1}$ for all $n \gg 0$.
\item 
For $n \gg 0$, the multiplicity of
$\Sp(\lambda[n])$ appearing
in $\widetilde{H}_q(X_n;\Q)$ grows as a quasi-polynomial in $n$.
\item The character of $\widetilde{H}_q(X_n,\Q)$ can be written:
\begin{align*}
  \chi_{\widetilde{H}_q(X_n,\Q)} =
  \sum_{A,\nu} c(\nu,A)\binom{\mathbf{X}}{\nu}A^{\mathbf{X}-\nu},
\end{align*}
where the constant $c(\nu,A) = 0$ whenever $|A| > 2^{q+1}$.
\item There exists an integer $e_{X,q}$,
depending only on the family $\{X_n\}_{n \in \N}$
and $q$, such that the torsion part of
the $\Z\symn$-module $\widetilde{H}_q(X_n,\Z)$
has exponent dividing $e_{X,q}$ for all $n\geq 0$.
\end{enumerate}
\end{theorem}

\begin{remark}
  The expression $\sum_{A,\nu} c(\nu,A)\binom{\mathbf{X}}{\nu}A^{\mathbf{X}-\nu}$
  is a genuinely infinite sum, though at any given conjugacy class all
  but finitely many terms are 0. Note also that if one takes the
  conjugacy class of the identity of $\Sym{n}$,
  then:
  $$\mathbf{X}_d(id_n) = \begin{cases}
        n &\text{ if $d = 1$}\\
        0 &\text{ otherwise.}
    \end{cases}$$
The character formula therefore simplifies to a $\Q[n]$ linear
  combination of exponential terms, recovering the first part of our theorem.
\end{remark}

We defer the proof of Theorem \ref{mainthm} to the next section
and explore the ramifications of the theorem in some examples.

If $X_n=X(n)=M(\Gamma)$, the matching complex of the complete
hypergraph on $n\geq 2$ vertices, then parts (i) and (ii) of
Theorem \ref{mainthm} are the observations made in the preamble
to this section. Parts (iii) and (iv) are less obvious, although
an alternative argument illustrates them in the case of the trivial
representation, where Frobenius reciprocity gives:
$$
\text{Hom}_{\symn}(1_n,V\kern1pt{\uparrow^{\kern1pt\Sym{n}}_{\kern1ptN_\lambda}})
\cong
\text{Hom}_{N_\lambda}(1_n\kern1pt{\downarrow^{\kern1pt\Sym{n}}_{\kern1ptN_\lambda}},V),
$$
with $1_n$ the trivial $\symn$-representation. As
$1_n\kern1pt{\downarrow^{\kern1pt\Sym{n}}_{\kern1ptN_\lambda}}$ is trivial and $V$
is irreducible, the right-hand side has dimension $1$ iff $V$ is the trivial
module, and this in turn happens exactly when $\lambda$ has no repeated parts.
Thus, the multiplicity of the trivial representation in
$V\kern1pt{\uparrow^{\kern1pt\Sym{n}}_{\kern1ptN_\lambda}}$ is
the number of partitions $\lambda$ of $n$ with $q+1$
distinct parts,  none of which have size $1$ or $n$. This quantity is thus
a quasi-polynomial when $n \gg 0$. This fact can also be deduced from
classical generating function arguments on collections of partitions.

Finessing this example a little further, we can produce
examples where the Betti number growth of
the family $\{X_n\}_{n \in \N}$ hits a variety of different types within
the description
given by Theorem \ref{mainthm}(i). For example,
let $X_n$ be just the $1$-skeleton of the complete hypergraph
matching complex $X(n)$. As this graph is connected we have
$H_0\cong\Z$ and so an Euler characteristic calculation gives
$1-\text{rk}\, H_1=|X_0|-|X_1|$, recalling that $X_q$ is
the set of $q$-faces. Thus:
$$
\text{rk}\, H_1=
\frac{1}{2}3^n-2\cdot2^n+\frac{7}{2},
$$
as $|X_0|=2^n-2$ and $|X_1|=\frac{3^n - (2^n+2^n-1)}{2}$.
If we instead take $X_n$ to be the complete graph on the vertex
set of $X(n)$---i.e. take the same vertices as $X(n)$ but now join any
two by an edge---then the rank of $H_1$
will have a $4^n$ term, thus meeting the worst-case scenario of
Theorem \ref{mainthm}(i).

As one final example of the various growth rates, consider our running example of 
the matching complex of the complete hypergraph. In this case, our theorem on the 
homologies of these spaces tells us that the rank of $H_1$ is equal to the number of 
partitions of $[n]$ into two proper non-singleton blocks. In particular:
\[
\text{rk}\ H_1 = \frac{2^n-2-2n}{2} = 2^{n-1}-(n+1).
\]
This example illustrates the genuine need for the polynomial coefficients when 
describing the Betti numbers of matching complexes, and fibre-closed families more 
generally.

Finally, we illustrate that it is even possible
for a fibre-closed family to exhibit torsion in its homology groups.
We define a fibre-closed family $\{X_n\}_{n \in \N}$ as follows.
For $n \leq 7$ let $X_n$ be the empty complex and let $X_7$ be the
matching complex of the complete graph $K_7$. Thus the vertices of $X_7$
consist of all pairs $\{i,j\}$, and a collection of vertices
forms a face if and only if none of the pairs share an element.
For $n > 7$, define $X_n$ inductively to be the smallest complex
containing all possible preimages of faces of $X_{n-1}$ under
surjections $f:[n] \rightarrow [n-1]$. By \cites{MR1253009,SW} we have
$H_1(X_7) \cong \Z/3\Z.$  The final part of Theorem \ref{mainthm}
tells us that in any fibre closed family of complexes, and for any
fixed $q \geq  0$, the $q$-th homology group $H_q(X_n)$
has torsion that is bounded independently of $n$.

\subsection{Representation stability of fibre-closed families}

The proof of Theorem \ref{mainthm} appeals to categorical representation
theory, and in particular the theory of $\FSop$-modules,
as found in \cites{PR-resonance,fs-braid,sam,Tost1,Tost2}.

Write $\FS$ for the category whose objects are the sets
$[n]=\{1,2,\ldots,n\}$, for $n \geq 0$,
and whose morphisms are all possible surjective maps.
An $\FSop$ module over a ring $R$ is a functor:
$$
V:\FSop \rightarrow\rmod,
$$
to the category of (left) $R$-modules. In other parts
of the literature
$\FSop$ modules are known variously
as presheaves, combinatorial sheaves and quiver representations. 

We will abbreviate the value of the functor $V$ on the object $[n]$
to just $V_n$. As
the surjective maps from $[n]$ to $[n]$ are
precisely the permutations of $[n]$,
an $\FSop$-module can also be thought of as a
collection of $R\symn$-modules $\{V_n\}_{n \geq 0}$,
such that for any surjection $[a] \rightarrow [b]$, there is
an induced map of $R$-modules $V_b \rightarrow V_a$ that
is compatible with the symmetric group actions.

The key idea of this section is that the homology groups, in fixed homological 
degree, of any fibre-closed family of simplicial complexes carry the structure of an 
$\FSop$-module. Indeed, given any surjection of finite sets, the action of taking 
fibres gives one a means for mapping faces of one member of the family to faces of 
another, thereby inducing a continuous map between the geometric realization of the 
complexes. For instance, returning to our running example of the matching complex of 
the complete hypergraph, we can actually determine the precise $\FSop$-module 
structure. Note that the action of a given surjection is precisely the most natural 
extension of the action (\ref{eq100}). Namely, given a surjection $f:[m] 
\twoheadrightarrow [n]$ and a facet of $X_n$ realized as a partition of $[n]$ into 
proper non-singleton blocks, take the preimages of each individual block. The result 
will once again be a partition of $[m]$ into proper non-singleton blocks. Having 
decided on an ordering on the blocks, one then adjusts by a sign, if necessary.

The category whose objects are the
$\FSop$-modules and whose morphisms are natural transformations of
functors is well known to be Abelian.
We say that an $\FSop$-module
$V$ is finitely generated if there is a finite
collection of elements $\{v_1,\ldots,v_r\} \subseteq \bigcup_n V_n$
which no proper $\FSop$-submodule of $V$ contains. Equivalently,
$V$ is finitely generated iff each $V_n$ is finitely generated as
an $R$-module
and there exists an integer $d \geq 0$ such that
for all $n > d$, the module $V_n$ is generated by the images of the
module $V_{n-1}$ under the maps induced by surjections $[n] \rightarrow [n-1]$.
In this case, we say that $V$ is finitely generated
in degrees $\leq d$. More generally, we say that an $\FSop$-module
is $d$-small if it is a subquotient of an $\FSop$-module
that is generated in degrees $\leq d$. 

We now summarise the results on $\FSop$-modules that we will need.
Let $R$ denote a Noetherian ring, and let $V$ be a
finitely generated $\FSop$-module over $R$ that is $d$-small.
Then by
\cite{sam}*{Corollary 8.1.3} all submodules of
$V$ are finitely generated. If $R=k$ is a field then 
there exist polynomials $f_1,\ldots,f_d \in \Q[x]$
such that for all $n \gg 0$:
$$
\dim V_n = f_1(n) + f_2(n)2^n + \ldots + f_d(n)d^n,
$$
--- see \cite{sam}*{Corollary 8.1.4}.

If $k$ is a field of
characteristic 0, then \cite{fs-braid}*{Theorem 4.1} gives that
for $n \gg 0$, all Specht modules $\Sp(\lambda)$
appearing as summands of $V_n$ are
associated to partitions $\lambda$ of length $\leq d$.
Moreover, if $\lambda$ is a partition of
some integer $m$, then for $n \gg 0$,
the muliplicity in $V_n$ of $\Sp(\lambda[n])$
is a quasi-polynomial in $n$;
see \cite{Tost2}*{Theorem 1.17}.
Finally, by \cite{Tost2}*{Theorem 1.11}
the character of $V_n$ can be expressed in terms of
the character polynomials $\mathbf{X}_d$ as shown above.

Our final ingredient is the following:

\begin{proposition}
If $R = \Z$, then there exists an integer
$e_V \geq 0$ such that for any $n$ the exponent of
torsion appearing in $V_n$ divides $e_V$.
\end{proposition}

\begin{proof}
Let $V$ be a finitely generated $\FSop$-module over $\Z$.
We observe that if $v \in V_n$ is any torsion element,
then the images of this element under any
map induced by surjections are also torsion.
It follows that we may define a
new $\FSop$-module by setting $W_n$ to be the collection of
all torsion elements of $V_n$ for each $n$.
Because $W$ is a submodule of $V$ by construction,
it must be finitely generated by the first result quoted above. 
The integer $d_V$ can then be taken to be
the least common multiple of the orders of the (finitely many)
generators of $W$.\qed
\end{proof}

We are now ready to proceed with the:

\begin{proof}[of Theorem \ref{mainthm}]
Let $\{X_n\}_{n \geq 0}$
be a fibre closed family of simplicial complexes.
For any $a > b \geq 0$, and any surjection $f:[a] \rightarrow [b]$,
we have a map $(X_b)_0 \rightarrow (X_a)_0$ by sending
a subset 
$x\subseteq [b]$ to its preimage under $f$.
The definition of fibre-closed tells us that this map lifts
to a map between the $q$-faces of $X_a$ and $X_b$ for all $q$.
In particular, the surjection $f$ induces a continuous map
between the complexes $f^\ast:X_b \rightarrow X_a$.
This map will descend to a map on homology, which gives us
the desired $\FSop$-structure. It only remains to show that
these $\FSop$-modules are finitely generated.

Define an $\FSop$-module $M$ over a Noetherian ring $R$ by
setting $M_n$ to be the $R$-linearization
of the power set of $[n]$. This $\FSop$-module is finitely
generated in degrees $\leq 2$. According to
\cite{PR-resonance}*{Lemma 2.1}, the tensor powers
$M^{\otimes (q+1)}$ --- defined in the obvious point-wise way
--- are finitely generated in degrees $\leq 2^{q+1}$.
It follows that the exterior power $\bigwedge^q M$ is also
generated in degrees $\leq 2^{q+1}$.

Next, for $q \geq 0$, define $C^{q}$ to be the $R$-module of
simplicial $q$-cochains of $X_n$ --- see
Section \ref{subsection:homology}.
For any $q$ the $\FSop$-module
$C^q$ is seen to be a submodule of $\bigwedge^{q+1} M$.
By the Noetherianity of $\FSop$-modules,
we conclude that $C^q$ is $2^{q+1}$-small, and therefore
the same can be said about the homology groups 
$\widetilde{H}_q(X_\bullet;R)$.
This concludes the proof.
  \qed
\end{proof}

%
%

%
%
%


\section*{References}


\begin{biblist}

\bib{MR1013569}{book}{
   author={Berge, Claude},
   title={Hypergraphs},
   series={North-Holland Mathematical Library},
   volume={45},
   note={Combinatorics of finite sets;
   Translated from the French},
   publisher={North-Holland Publishing Co., Amsterdam},
   date={1989},
   pages={x+255},
   isbn={0-444-87489-5},
}

\bib{MR1333388}{article}{
   author={Bj\"{o}rner, Anders},
   author={Wachs, Michelle L.},
   title={Shellable nonpure complexes and posets. I},
   journal={Trans. Amer. Math. Soc.},
   volume={348},
   date={1996},
   number={4},
   pages={1299--1327},
   issn={0002-9947},
}

\bib{MR1253009}{article}{
   author={Bj\"{o}rner, A.},
   author={Lov\'{a}sz, L.},
   author={Vre\'{c}ica, S. T.},
   author={\v{Z}ivaljevi\'{c}, R. T.},
   title={Chessboard complexes and matching complexes},
   journal={J. London Math. Soc. (2)},
   volume={49},
   date={1994},
   number={1},
   pages={25--39},
   issn={0024-6107},
 }

 \bib{MR1174893}{article}{
   author={Bouc, S.},
   title={Homologie de certains ensembles de $2$-sous-groupes des groupes
   sym\'{e}triques},
   language={French, with French summary},
   journal={J. Algebra},
   volume={150},
   date={1992},
   number={1},
   pages={158--186},
   issn={0021-8693},
}

\bib{MR0385008}{article}{
   author={Brown, Kenneth S.},
   title={Euler characteristics of groups: the $p$-fractional part},
   journal={Invent. Math.},
   volume={29},
   date={1975},
   number={1},
   pages={1--5},
   issn={0020-9910},
}

 \bib{MR0460128}{book}{
   author={Comtet, Louis},
   title={Advanced combinatorics},
   edition={enlarged edition},
   note={The art of finite and infinite expansions.},
   publisher={D. Reidel Publishing Co., Dordrecht},
   date={1974},
   pages={xi+343},
   isbn={90-277-0441-4},
}
 
 

\bib{MR1912799}{article}{
   author={Dong, Xun},
   author={Wachs, Michelle L.},
   title={Combinatorial Laplacian of the matching complex},
   journal={Electron. J. Combin.},
   volume={9},
   date={2002},
   number={1},
   pages={Research Paper 17, 11},
   doi={10.37236/1634},
}

\bib{donkin}{article}{
    author = {Donkin, Stephen},
     title = {On tilting modules for algebraic groups},
   journal = {Math. Z.},
    volume = {212},
      date = {1993},
    number = {1},
     pages = {39--60},
      issn = {0025-5874,1432-1823},
}

\bib{MR1648484}{article}{
   author={Friedman, Joel},
   author={Hanlon, Phil},
   title={On the Betti numbers of chessboard complexes},
   journal={J. Algebraic Combin.},
   volume={8},
   date={1998},
   number={2},
   pages={193--203},
   issn={0925-9899},
   doi={10.1023/A:1008693929682},
}
 
 \bib{MR0600368}{article}{
   author={Howard, F. T.},
   title={Associated Stirling numbers},
   journal={Fibonacci Quart.},
   volume={18},
   date={1980},
   number={4},
   pages={303--315},
   issn={0015-0517},
}

\bib{MR2368284}{book}{
   author={Jonsson, Jakob},
   title={Simplicial complexes of graphs},
   series={Lecture Notes in Mathematics},
   volume={1928},
   publisher={Springer-Verlag, Berlin},
   date={2008},
   pages={xiv+378},
   isbn={978-3-540-75858-7},
}

\bib{MR2536865}{book}{
   author={Lov\'{a}sz, L\'{a}szl\'{o}},
   author={Plummer, Michael D.},
   title={Matching theory},
   note={Corrected reprint of the 1986 original [MR0859549]},
   publisher={AMS Chelsea Publishing, Providence, RI},
   date={2009},
   pages={xxxiv+554},
   isbn={978-0-8218-4759-6},
}

\bib{MR1402473}{book}{
   author={Maunder, C. R. F.},
   title={Algebraic topology},
   note={Reprint of the 1980 edition},
   publisher={Dover Publications, Inc., Mineola, NY},
   date={1996},
   pages={viii+375},
   isbn={0-486-69131-4},
}

\bib{PR-resonance}{article}{
   author={Proudfoot, Nicholas},
   author={Ramos, Eric},
   title={Stability phenomena for resonance arrangements},
   journal={Proc. Amer. Math. Soc. Ser. B},
   volume={8},
   date={2021},
   pages={219--223},
   doi={10.1090/bproc/71},
}

\bib{fs-braid}{article}{
   author={Proudfoot, Nicholas},
   author={Young, Ben},
   title={Configuration spaces, $\rm FS^{op}$-modules,
     and Kazhdan-Lusztig
   polynomials of braid matroids},
   journal={New York J. Math.},
   volume={23},
   date={2017},
   pages={813--832},
}

 \bib{MR0493916}{article}{
   author={Quillen, Daniel},
   title={Homotopy properties of the poset of
     nontrivial $p$-subgroups of a
   group},
   journal={Adv. in Math.},
   volume={28},
   date={1978},
   number={2},
   pages={101--128},
   issn={0001-8708},
 }

 \bib{MR1824028}{book}{
   author={Sagan, Bruce E.},
   title={The symmetric group},
   series={Graduate Texts in Mathematics},
   volume={203},
   edition={2},
   note={Representations, combinatorial algorithms,
     and symmetric functions},
   publisher={Springer-Verlag, New York},
   date={2001},
   pages={xvi+238},
   isbn={0-387-95067-2},
 }

 \bib{sam}{article}{
   author={Sam, Steven V.},
   author={Snowden, Andrew},
   title={Gr\"{o}bner methods for representations of combinatorial
   categories},
   journal={J. Amer. Math. Soc.},
   volume={30},
   date={2017},
   number={1},
   pages={159--203},
   issn={0894-0347},
   doi={10.1090/jams/859},
}

\bib{SW}{article}{,
  title={Torsion in the matching complex and chessboard complex},
  author={Shareshian, John},
  author = {Wachs, Michelle L},
  journal={Advances in Mathematics},
  volume={212},
  number={2},
  pages={525--570},
  year={2007},
  publisher={Elsevier}
}

\bib{MR0210112}{book}{
   author={Spanier, Edwin H.},
   title={Algebraic topology},
   publisher={McGraw-Hill Book Co., New York-Toronto, Ont.-London},
   date={1966},
   pages={xiv+528},
}

\bib{Tost1}{article}{
   author={Tosteson, Philip},
   title={Categorifications of rational
     Hilbert series and characters of
   $FS^{\rm op}$ modules},
   journal={Algebra Number Theory},
   volume={16},
   date={2022},
   number={10},
   pages={2433--2491},
   issn={1937-0652},
   doi={10.2140/ant.2022.16.2433},
}

\bib{Tost2}{article}{
   author={Tosteson, Philip},
   title={Stability in the homology of
     Deligne-Mumford compactifications},
   journal={Compos. Math.},
   volume={157},
   date={2021},
   number={12},
   pages={2635--2656},
   issn={0010-437X},
   review={\MR{4354696}},
   doi={10.1112/s0010437x21007582},
}

 \bib{Wachs07}{article}{
    author={Wachs, Michelle L.},
    title={Poset topology: tools and applications},
    conference={
       title={Geometric combinatorics},
    },
    book={
       series={IAS/Park City Math. Ser.},
       volume={13},
       publisher={Amer. Math. Soc.},
       place={Providence, RI},
    },
    date={2007},
    pages={497--615},
 }

\bib{MR2022345}{article}{
   author={Wachs, Michelle L.},
   title={Topology of matching, chessboard,
     and general bounded degree graph
   complexes},
   note={Dedicated to the memory of Gian-Carlo Rota},
   journal={Algebra Universalis},
   volume={49},
   date={2003},
   number={4},
   pages={345--385},
   issn={0002-5240},
}

\end{biblist}
\end{document}